\documentclass[letter,11pt,reqno]{amsart}

\usepackage[utf8]{inputenc}

\usepackage{etex}

\usepackage{xcolor}

\definecolor{verydarkblue}{rgb}{0,0,0.5}

\usepackage[
    breaklinks,
    colorlinks,
    citecolor=verydarkblue,
    linkcolor=verydarkblue,
    urlcolor=verydarkblue,
    pagebackref=true,
    hyperindex
]{hyperref}

\backrefenglish

\usepackage{fancyhdr}

\usepackage[
    hscale=0.7,
    vscale=0.75,
    headheight=13pt,
    centering,
]{geometry}

\usepackage{amsmath}
\usepackage{amsthm}
\usepackage{amssymb}
\usepackage{mathtools}
\usepackage{bm}		%for boldface greek symbols
\usepackage{mathdots}
\usepackage{framed}
\usepackage[capitalize]{cleveref}
\usepackage{array}
\usepackage[alphabetic,msc-links]{amsrefs}
\usepackage[all,cmtip]{xy}

\theoremstyle{plain}

\newtheorem{introtheorem}{Theorem}

\crefname{introtheorem}{Theorem}{Theorems}

\crefname{introcorollary}{Corollary}{Corollaries}

\newtheorem{theorem}{Theorem}[section]
\newtheorem*{theorem*}{Theorem}
\newtheorem{proposition}[theorem]{Proposition}
\newtheorem{lemma}[theorem]{Lemma}
\newtheorem{corollary}[theorem]{Corollary}

\theoremstyle{definition}

\newtheorem{definition}[theorem]{Definition}

\theoremstyle{remark}

\newtheorem{remark}[theorem]{Remark}
\newtheorem{example}[theorem]{Example}

\numberwithin{figure}{section}

\numberwithin{equation}{section}

\usepackage{marginnote}

\def\Z{{\mathbb Z}}
\def\Q{{\mathbb Q}}
\def\R{{\mathbb R}}
\def\C{{\mathbb C}}

\def\A{{\mathbb A}}
\def\P{{\mathbb P}}

\def\I{\mathcal{I}}

\def\O{\mathcal{O}}

\def\fa{\mathfrak{a}}
\def\fb{\mathfrak{b}}

\def\fm{\mathfrak{m}}

\def\a{\alpha}
\def\b{\beta}
\def\g{\gamma}

\def\ff{\psi}
\def\e{\eta}

\def\n{\nu}
\def\m{\mu}

\def\p{\pi}

\def\D{\Delta}
\def\G{\Gamma}
\def\S{\Sigma}
\def\Om{\Omega}

\def\.{\cdot}
\let\circum\^
\def\^{\widehat}
\def\~{\widetilde}

\def\ov{\overline}

\def\({\left(}
\def\){\right)}

\def\*{{}^*}

\renewcommand{\and}{ \ \ \text{ and } \ \ }

\def\reg{\mathrm{reg}}

\def\red{\mathrm{red}}

\def\Jac{\mathrm{Jac}}

\DeclareMathOperator{\codim} {codim}

\DeclareMathOperator{\Spec} {Spec}

\DeclareMathOperator{\Sing} {Sing}

\DeclareMathOperator{\val} {val}

\DeclareMathOperator{\mult} {mult}
\DeclareMathOperator{\ord} {ord}

\DeclareMathOperator{\Sym} {Sym}

\DeclareMathOperator{\lct} {lct}

\DeclareMathOperator{\jetcodim} {jet-codim}

\DeclareMathOperator{\mld} {mld}

\def\embdim{\mathrm{edim}}
\def\embcodim{\mathrm{ecodim}}

\def\jetcodim{\mathrm{jet.codim}}

\begin{document}

\title
[Families of jets of arc type]
{Families of jets of arc type and higher (co)dimensional Du Val singularities}

\author{Tommaso de Fernex}
\address{Department of Mathematics, University of Utah, Salt Lake City, UT 84112, USA}
\email{{\tt defernex@math.utah.edu}}

\author{Shih-Hsin Wang}
\address{Department of Mathematics, University of Utah, Salt Lake City, UT 84112, USA}
\email{{\tt shwang@math.utah.edu}}

\dedicatory{In memory of Jean-Pierre Demailly}


\thanks{%
Research of the first author partially supported by NSF Grant DMS-2001254.
}

\begin{abstract}
Families of jets through singularities of algebraic varieties are here 
studied in relation to the families of arcs originally studied by Nash.
After proving a general result relating them, 
we look at normal locally complete intersection varieties with rational singularities 
and focus on a class of singularities we call \emph{higher Du Val singularities}, 
a higher dimensional (and codimensional) version of Du Val singularities that is
closely related to Arnold singularities.
More generally, we introduce the notion of \emph{higher compound Du Val singularities}, 
whose definition parallels that of compound Du Val singularities.
For such singularities, we prove that there exists
a one-to-one correspondence between families of arcs
and families of jets of sufficiently high order through the singularities.
In dimension two, the result partially recovers a theorem of Mourtada
on the jet schemes of Du Val singularities.
As an application, we give a solution of the Nash problem for 
higher Du Val singularities. 
\end{abstract}

\maketitle

\section{Introduction}

The space of arcs through 
the singular locus of a complex variety decomposes into a finite union
of irreducible components, each defining a distinct divisorial
valuation, that is, a prime divisor on some resolution of singularities.
These components were studied by Nash \cite{Nas95};
we will refer to them as \emph{Nash families of arcs},
and to the valuations they define as \emph{Nash valuations}.
The problem of characterizing Nash families of arcs in terms of resolutions
of singularities fits within the \emph{Nash problem}, which was motivated by the desire of 
understanding what different resolutions would have in common.

It is natural to ask whether a similar picture
holds for families of jets through the singular locus, 
at least when one looks at jets of sufficiently high order. 
(For clarity of exposition, in this introduction we restrict the discussion to the case 
where families of arcs and families of jets all stem from the singular locus of the variety; 
we refer to the main body of the paper for a more general formulation of the question.)
As jets are parametrized by schemes of finite type, the fact that there are finitely many
irreducible components of the set of jets of fixed order through the singular locus
is clear. The question is how the families of 
jets defined by such components relate to the families of arcs identified by Nash. 

Even though families of jets are introduced similarly to families of arcs, 
at the core there is an essential difference between the two: 
Nash families of arcs lift to resolutions of singularities and are naturally related to
divisorial valuations; by contrast, families of jets through singularities do not lift to resolutions 
and cannot be related to valuations in any obvious way.
In particular, the approach followed by Nash to study families of arcs using resolution of singularities
does not apply to finite order jets.

Families of jets have been computed in several concrete examples, see, e.g., 
the works on plane curves and surface singularities 
\cite{Mou11,Mou11b,LJMR13,Mou14,Mou17,MC18,CM21}; 
in many of these works, the computation is carried out through a direct analysis of the 
defining equations.
The problem of understanding families of jets
is closely related to the \emph{embedded Nash problem},
which aims to describe the irreducible components
of contact loci of effective divisors on smooth ambient varieties
in terms of embedded resolutions.
A breakthrough in this direction was recently made in \cite{BdlBdLFdBP22}, 
where the problem was solved for unibranched plane curves;
see also, e.g., \cite{Ish08,FdBPPPP17} for earlier
work on this problem.

The purpose of this paper is to unveil
a natural correspondence between families of arcs
and certain families of jets of sufficiently high order.
Our starting point is the following general property.

\begin{introtheorem}[\cref{t:welldefined-injective}]
\label{t:intro:welldefined-injective}
Among all families of jets of sufficiently high order stemming 
the singular locus of a variety, there is a selection of them
that is in natural one-to-one correspondence with the Nash families of arcs. 
\end{introtheorem}

The correspondence is obtained by defining, in a geometric meaningful way, an injective map
from the set of Nash families of arcs to the set of families of jets through the singular locus.
We say that a family of jets is \emph{of arc type} if it is in the image of this map.

We then address the question whether
all families of jets of sufficiently high order through the singular locus are of arc type. 
Although in general there are more families of jets compared to families of arcs
(see, e.g., the case of toric surface singularities \cite{Mou11b,Mou17}), 
we will show that there is a one-to-one correspondence 
for certain rational singularities of arbitrary dimensions. 
One case we already understand, thanks to \cite{Mou14}, is that of Du Val singularities,
where there is a one-to-one correspondence.
Here we extend the existence of such correspondence to
a large class of locally complete intersection rational singularities of arbitrary dimensions
which include isolated compound Du Val singularities. 

For every normal locally complete intersection variety $X$ 
there is a bound on embedding codimension in terms of 
minimal log discrepancy. The bound, which is proved in \cref{p:mld-bound}, is given by
\[
\embcodim(\O_{X,x}) \le \dim(\O_{X,x}) - \mld_x(X)
\] 
for every $x \in X$. We say that $X$ has \emph{maximal embedding codimension} at $x$
if the bound is achieved.
Within this class of singularities, we have those for which
\[
\mld_x(X) = \dim(\O_{X,x}) - \embcodim(\O_{X,x}) = 1.
\]
It is easy to see that these are isolated singularities.
We will see in a moment that these singularities have many properties
that are natural higher dimensional analogues of properties characterizing
Du Val singularities in dimension two (the analogy is also manifest
in the examples provided in \cref{p:eg-hDV}). For this reason, 
we call these singularities \emph{higher Du Val singularities}.
In dimension two, this class of singularities coincides with Du Val singularities.

We then look at rational singularities of maximal embedding dimension
that reduce to higher Du Val singularities 
under generic hyperplane sections. One should think of this condition as
an analogue of the definition of compound Du Val singularity. 
We call these singularities 
\emph{higher compound Du Val singularities}. 
We have the following result.

\begin{introtheorem}[\cref{t:cDV}]
\label{t:intro:cDV}
On an isolated higher compound Du Val singularity $x \in X$, 
all families of jets of sufficiently high order stemming from $x$ are of arc type. 
\end{introtheorem}

As a special case, we see that all families of jets of sufficiently high order stemming from 
an isolated compound Du Val singularities are of arc type. 
\cref{t:intro:cDV} addresses our motivating question on families of jets. 
Combined with \cref{t:intro:welldefined-injective}, the theorem
relates to and partially recover a result of Mourtada on
families of jets on Du Val singularities \cite{Mou14} (see \cref{c:graph-crepant-sing}).
Mourtada asked whether for any locally complete intersection variety with rational
singularities the number of families of jets of sufficiently high order stemming from 
the singular locus is the same as the number of Nash families of arcs
\cite[Question~4.5]{Mou14}. Our result provides evidence in this direction. 

For higher Du Val singularities, we have a
more precise result (see \cref{t:surjective}) which we use to solve 
the Nash problem for this class of singularities.
In our solution, Nash valuations are characterized in terms of 
certain partial resolutions of the variety (the terminal models)
that originate from the minimal model program.
Valuations defined by the exceptional divisors on these models are called \emph{terminal valuations}.

\begin{introtheorem}[\cref{c:Nash-crepant-sing}]
\label{t:intro:Nash-crepant-sing}
For a divisorial valuation $\ord_E$ on a variety $X$
with higher Du Val singularities, the following are equivalent:
\begin{enumerate}
\item
$\ord_E$ is a Nash valuation.
\item
$\ord_E$ is a terminal valuation.
\item
$E$ is a crepant exceptional divisor over $X$.
\end{enumerate}
\end{introtheorem}

This result is in line with the point of view proposed in \cite{dFD16}.
It can be viewed as a higher dimensional generalization of one of the properties
characterizing Du Val singularities
among normal surface singularities.

In dimension two, there are four proofs of the Nash 
problem for Du Val singularities \cite{PP13,FdBPP12,Reg12,dFD16}.
While the proof given here follows a different path,
relying on inversion of adjunction and the minimal model program, 
it also uses on the main theorem of \cite{dFD16}
and therefore it should not be considered as providing a new proof in dimension two
for Du Val singularities.
In higher dimensions, however, \cref{t:intro:Nash-crepant-sing} does not
follow directly from \cite{dFD16}.

Throughout the paper, we work over an algebraically closed field $k$ of characteristic zero. 

\subsection*{Acknowledgments} 

We thank Roi Docampo, Hsueh-Yung Lin, and the referee for valuable comments and suggestions.

\section{Arc spaces and jet schemes}

For a scheme $X$ over $k$, we 
denote by $X_\infty$ the \emph{arc space} of $X$ over $k$ 
and by $X_m$ the \emph{$m$-th jet scheme} of $X$. 
We refer to \cite{Voj07,EM09,dF18} for general references on the subject.
An arc $\a \in X_\infty$ is a morphism $\a \colon \Spec k_\a[[t]] \to X$
and a jet $\b \in X_m$ a morphism $\b \colon \Spec k_\b[t]/(t^{m+1}) \to X$.
We denote by $\a(0)$ and $\b(0)$ the images of the respective closed points, 
and by $\a(\e)$ the image of the generic point of $\Spec k_\a[[t]]$.
There are truncation maps $\p \colon X_\infty \to X$ and 
$\p_m \colon X_m \to X$ sending an arc $\a$
(respectively, an $m$-jet $\b$) to its special point $\a(0)$ (respectively, $\b(0)$), 
as well as $\ff_m \colon X_\infty \to X_m$ and $\p_{n,m} \colon X_n \to X_m$ for $n > m$. 
We denote these maps by $\p^X$, $\p_m^X$, $\ff_m^X$, and $\p_{n,m}^X$
whenever there is a need to specify the underlying scheme $X$. 

Let now $X$ be a variety. 
Constructibility in $X_\infty$ is defined as in \cite{EGAiii_1} 
(see also \cite[\href{https://stacks.math.columbia.edu/tag/005G}{Tag~005G}]{SP}):
a subset $C \subset X_\infty$ is \emph{constructible} if and only if it is a finite union of
finite intersections of retrocompact open sets and their complements;
equivalently, $C$ is constructible if and only $C = \ff_m^{-1}(S)$
for some $m$ and some constructible set $S \subset X_m$. 
An irreducible subset $C \subset X_\infty$ is \emph{non-degenerate}
if $C \not\subset (\Sing X)_\infty$.

When $X$ is smooth, constructible sets are also called \emph{cylinders}. 
Their \emph{codimension} is defined by 
$\codim(C,X_\infty) := \codim(S,X_m)$ where, as before, $C = \ff_m^{-1}(S)$. 
Using the simple structure of the truncation maps $\p_{n,m}$, 
it is easy to check that this is well defined. 
The codimension of $C$ defined above agrees with 
topological codimension of the closure of $C$ in $X_\infty$;
if $C$ is irreducible and $\a \in C$ is the generic point, then
this is the same as $\dim(\O_{X_\infty,\a})$. 

When $X$ is singular, one defines the \emph{jet codimension}
of a constructible set $C \subset X_m$ by setting
$\jetcodim(C,X_\infty) := (m+1)\dim(X) - \dim(S)$ where, again, $C = \ff_m^{-1}(S)$
(cf.\ \cite{dFEI08}). 
If $C$ is irreducible and $\a \in C$ is the generic point, then this agrees
with $\embdim(\O_{X_\infty,\a})$.

Every arc $\a \in X_\infty$ defines a semi-valuation 
$\ord_\a \colon \O_{X,\a(0)} \to \Z \cup \{\infty\}$,
given by $\ord_\a(h) = \ord_t(\a^\sharp(h))$,
which extends to a valuation of the function field of $X$
if and only if the generic point $\a(\e)$ of the arc is the generic point of $X$. 
In a similar fashion, every jet $\b \in X_m$
defines a function $\ord_\b \colon \O_{X,\b(0)} \to \{0,1,\dots,m\} \cup \{\infty\}$
given by $\ord_\b(h) = \ord_t(\b^\sharp(h))$, where we set $\ord_t(0) = \infty$.

A \emph{prime divisor} over $X$ is, by definition,
a prime divisor $E$ on a normal birational model $Y \to X$. 
Any such divisor $E$ defines a valuation $\ord_E$ on $X$.
A valuation on $X$ of the form $v = q\ord_E$
where $E$ is a prime divisor over $X$ and $q$ is a positive integer
is called a \emph{divisorial valuation}. 
The image in $X$ of the generic point of $E$ is called the 
\emph{center} of the valuation (or of $E$), and is denoted by $c_X(v)$ or $c_X(E)$.
For a divisorial valuation $v = q\ord_E$, the closure $C_X(v) \subset X_\infty$
of the set of arcs $\a$ such that $\ord_\a = v$
is called the \emph{maximal divisorial set} associated to $v$.
This is an irreducible closed constructible subset of $X_\infty$.
When $v = \ord_E$, we also denote this set by $C_X(E)$. 

Let now $X$ be a variety.
As shown in \cite{Nas95} (see also, e.g., \cite{IK03,dF18}), the set
$\p^{-1}(\Sing X)$ decomposes as a finite union of irreducible components, and each component defines a
divisorial valuation on $X$. These are called \emph{Nash valuations}
and the problem is to characterize them. Nash conjectured that, in dimension two, Nash valuations 
are precisely those defined by the exceptional divisors on the minimal resolution, and proposed the notion of
\emph{essential divisor} as a possible higher dimensional generalization which he speculated may 
characterize Nash valuations in all dimensions. 
These questions, which are generally referred to as the \emph{Nash problem}, have generated a lot of activity. 

Culminating the work of many people, the
complete solution of the Nash problem in dimension two was eventually found by Fernandez de Bobadilla
and Pe Pereira in \cite{FdBPP12}, and before that, in the toric case by Ishii and Koll\'ar \cite{IK03}. 
A new, algebraic proof in the surface case was later found in \cite{dFD16}, 
where it was proved that, in any dimension, all valuations defined by exceptional divisors on 
terminal models over $X$ are Nash valuations; we call the valuations arising in this way
\emph{terminal valuations}. Nash's original guess of
what the picture should be in dimension $\ge 3$, however, turned out
to be incorrect \cite{IK03,dF13,JK13}.
In view of this, one can reinterpret the Nash problem as asking
for a characterization of Nash valuations 
in terms of resolution of singularities of a variety $X$ and, more generally, its birational geometry.

\section{Minimal log discrepancies}

Let $X$ be a normal variety, and assume that its canonical class $K_X$ is $\Q$-Cartier.
For every prime divisor $E$ over $X$, if $f \colon Y \to X$ is the normal birational
model where $E$ lies, then we define the \emph{log discrepancy}
of $E$ over $X$ by $a_E(X) := \ord_E(K_{Y/X}) + 1$, and
the \emph{Mather log discrepancy}
of $E$ over $X$ by $\^a_E(X) := \ord_E(\Jac_f) + 1$. 
These invariants of $E$ only depends on the valuation $\ord_E$, 
and they agree if $X$ is smooth at the center of $E$.

An \emph{effective $\R$-subscheme} $Z$ of $X$ is an expression 
$Z = \sum_{j=1}^s c_jZ_j$ where $Z_j \subset X$ is a proper closed subscheme 
and $c_j > 0$ for every $j$. Its \emph{support} is 
the union of the support of the $Z_j$.
For any effective $\R$-subscheme $Z$, we define the \emph{log discrepancy}
of $E$ over the pair $(X,Z)$ to be $a_E(X,Z) := a_E(X) - \sum c_j \ord_E(\I_{Z_j})$
where $\I_{Z_j} \subset \O_X$ is the ideal sheaf of $Z_j$.
The \emph{minimal log discrepancy} of $(X,Z)$ at a
point $x$ is defined by
\[
\mld_x(X,Z) := \inf_{c_X(E) = x} a_E(X,Z)
\]
where the infimum is taken over all prime divisors $E$ with center $x$.
When there is no $Z$, we just write $\mld_x(X)$. 
We set $\mld_x(X,Z) = 0$ if $x$ is the generic point of $X$.
If $\dim X \ge 2$, then $\mld_x(X,Z) \in \{- \infty\} \cup [0, \infty)$.
For sake of uniformity, it is convenient 
to declare that $\mld_x(X,Z) = - \infty$ whenever it is negative when $\dim X = 1$ as well.

The following is a slightly more general reformulation of the main theorem of \cite{EM04}. 
The proof is essentially contained in \cite{EM09}. We review the key part of the argument
for completeness. A similar argument will also be used later in the paper, so it is useful to review
it here anyway.

\begin{theorem}[Inversion of adjunction \cite{EM04}]
\label{t:inv-adj}
Let $X$ be a smooth variety, $Y = H_1 \cap \dots \cap H_e \subset X$ a normal positive dimensional
subvariety defined by the complete intersection of $e$ hypersurfaces $H_i \subset X$, and
$Z = \sum c_jZ_j$ an effective $\R$-subscheme of $X$ not containing $Y$ in its support.
Then for every $x \in Y$ we have
\[
\mld_x(Y,Z|_Y) = \mld_x(X,Z + eY) = \mld_x\Big(X,Z + \sum_{i=1}^e H_i\Big),
\] 
where $Z|_Y := \sum c_j(Z_j \cap Y)$. 
\end{theorem}

\begin{proof}
We may assume that $x$ is not the generic point of $Y$, the statement being elementary in that case. 
The proofs of the inequalities $\mld_x(Y,Z|_Y) \ge \mld_x(X,Z + eY) \ge \mld_x(X,Z + \sum H_i)$
are fairly standard and are omitted.  We review the proof of the inequality
\[
\mld_x(Y,Z|_Y) \le \mld_x(X,Z + \sum H_i),
\] 
which is the hard part of the theorem.
To this end, it suffices to show that
for every divisorial valuation $v = \ord_F$ on $X$ with center $c_X(v) = x$, 
there is a divisorial valuation $w = q\ord_E$ over $Y$ with center $c_Y(w) = x$
such that 
\[
q\,a_E(Y,Z|_Y) \le a_F(X,Z + \sum H_i).
\]

We denote by $Y_\infty^x$ the reduced inverse image
of $x$ under the projection $\p^Y \colon Y_\infty \to Y$. 
By definition, $Y_\infty^x$ is the set of arcs in $Y$ stemming from $x$. 

Let $C_X(v) \subset X_\infty$ 
be the maximal divisorial set associated to $v$.
Note that $\p^X(C_X(v))$ is an irreducible constructible set
with generic point $x$. Consider the intersection
\[
C_X(v) \cap Y_\infty.
\]
As $v$ is centered at $x$ and $C_X(v)$ is closed under the action of 
the morphism $\Phi_\infty \colon \A^1 \times X_\infty \to X_\infty$ 
given by $(a,\a(t)) \mapsto \a(at)$ (cf.\ \cite[Section~3]{EM04}), 
we see that $C_X(v)$ contains the constant arc at $x$, 
hence $C_X(v) \cap Y_\infty^x \ne \emptyset$.
It follows that $x$ is the generic point of $\p^X(C_X(v) \cap Y_\infty)$.
Therefore we can pick an irreducible component $W$ of $C_X(v) \cap Y_\infty$
such that $\p^Y(W)$ has $x$ as its generic point.
Note that \cite[Lemma~8.3]{EM09} applies to $C_X(v) \cap Y_\infty^x$
since both $C_X(v)$ and $Y_\infty^x$ are closed under the action of 
the morphism $\Phi_\infty$, hence 
$C_X(v) \cap Y_\infty^x \not\subset (\Sing Y)_\infty$.
Therefore we can assume that $W$ is not contained in $(\Sing Y)_\infty$. 
By construction $W$ is the closure of an irreducible constructible set in $Y_\infty$, 
hence, by \cite{dFEI08}, its generic point $\g \in W$ defines a divisorial valuation
$w = q\ord_E$ on $Y$, and \cite[Lemma~8.4]{EM09} (its proof, to be precise) gives
\[
\jetcodim(W,Y_\infty) \le \codim(C_X(v),X_\infty) + q\ord_E(\Jac_Y) - \sum \ord_F(\I_{H_i}).
\]
Since $W \subset C_Y(w)$, \cite[Theorem~3.8]{dFEI08} implies that
$\jetcodim(W,Y_\infty) \ge q\.\^a_E(Y)$
where $\^a_E(Y)$ is the Mather log discrepancy.
As $Y$ is normal and locally complete intersection, we
have $\^a_E(Y) = a_E(Y) + \ord_E(\Jac_Y)$ (see, e.g., \cite[Corollary~3.5]{dFD14}), 
hence 
\[
\jetcodim(W,Y_\infty) \ge q (a_E(Y) + \ord_E(\Jac_Y)).
\]
On the other hand, as $X$ is smooth, we have
\[
\codim(C_X(v),X_\infty) = a_F(X).
\]
Finally, by the semicontinuity of order of contact function induced by $\I_{Z_j}$ on $X_\infty$, we have
\[
q\ord_E(\I_{Z_j}\.\O_Y) \ge \ord_F(\I_{Z_j}).
\]
By combining the above formulas, we conclude that $q\,a_E(Y,Z|_Y) \le a_F(X,Z + \sum H_i)$.
\end{proof}

\begin{remark}
\label{r:inv-adj=}
Going through the above proof (with $Z = 0$), suppose that 
$a_F(X,\sum H_i) = \mld_x(X,\sum H_i) \ge 0$.
Then we necessarily have $q\,a_E(Y) = a_F(X,\sum H_i)$, 
since $q\,a_E(Y) \ge a_E(Y) \ge \mld_x(Y)$,
hence $q\,a_E(Y) = a_E(Y) = \mld_x(Y)$.
In particular, if $\mld_x(Y) > 0$ then $q = 1$.
Furthermore, the inequalities in the formulas displayed in the proof
must all be equalities, hence $W = C_Y(w)$.
\end{remark}

\begin{corollary}
\label{c:inv-adj}
Let $X$ be a normal locally complete intersection variety, 
$Y = H_1 \cap \dots \cap H_e \subset X$ a normal positive dimensional
subvariety defined by the complete intersection of $e$ hypersurfaces $H_i \subset X$, and
$Z = \sum c_jZ_j$ an effective $\R$-subscheme of $X$ not containing $Y$ in its support.
Then for every $x \in Y$ we have
\[
\mld_x(Y,Z|_Y) = \mld_x(X,Z + eY) = \mld_x\Big(X,Z + \sum_{i=1}^e H_i\Big).
\] 
\end{corollary}

\begin{proof}
Again, it suffices to prove that $\mld_x(Y,Z_Y) = \mld_x(X,Z + \sum H_i)$. 
Working locally near $x$, we can fix a closed embedding $X \subset A$ where $A$ is a smooth
variety, and hypersurfaces $D_1, \dots, D_r \subset A$ where $r = \codim(Y,A)$, 
such that $H_i = D_i \cap X$ for $i=1,\dots,e$ and $X = D_{e+1} \cap \dots \cap D_r$.
Note that $Y = D_1 \cap \dots \cap D_r$. 
By \cref{t:inv-adj} (applied twice, to $Y \subset A$ and $X \subset A$), we have
\[
\mld_x(Y,Z|_Y) 
= \mld_x\Big(A,Z + \sum_{i=1}^r D_i\Big) 
= \mld_x\Big(X,Z + \sum_{i=1}^e H_i\Big).
\]
This completes the proof.
\end{proof}

\section{Families of jets of arc type}

Let $X$ be a positive dimensional variety.
For any subset $\S \subset X$, we consider the sets
\[
X_\infty^\S := \p^{-1}(\S)_\red = \{\a \in X_\infty \mid \a(0) \in \S\}
\]
and
\[
X_m^\S := \p_m^{-1}(\S)_\red = \{\b \in X_m \mid \b(0) \in \S\}.
\]
By definition, $X_\infty^\S$ is the set of arcs on $X$ through $\S$,
and $X_m^\S$ is the set of $m$-jets through $\S$.  

Assume that $\S \subset X$ is a closed subset.
Since $X_m$ is a scheme of finite type, each $X_m^\S$ decomposes into 
a finite union of irreducible components, and a generalization of 
Nash's theorem \cite{Nas95} tells us that the same happens for $X_\infty^\S$. 

In the following, we denote by $\G \subset X_\infty^\S \setminus (\Sing X)_\infty$ 
the set of generic points; that is, $\a \in \G$
if and only if $\a$ is the generic point of a non-degenerate irreducible component of $X_\infty^\S$.
Let 
\[
\m := \max_{\a\in \G} \ord_\a(\Jac_X).
\]
Note that $\m < \infty$ since $\G$ is finite and each $\a\in\G$ is non-degenerate.

We fix an integer $\n \ge \m$ such that the images $\ff_\n(\a) \in X_\n$, 
for $\a \in \G$, are all distinct and
there are no specializations within the set $\ff_\n(\G) \subset X_\n$
(meaning that $\ff_\n(\G)$, with the induced topology, is discrete). 
The existence of such integer follows from the definition of $X_\infty$
as inverse image of the jet schemes under the truncation maps.

\begin{theorem}
\label{t:welldefined-injective}
Let $X$ be a variety and $\S \subset X$ a closed subset. 
Then for every $m \ge \mu + \nu$ there is a naturally defined injective map
\[
\Psi^\S_m \colon \{\text{non-degenerate irreducible components of $X_\infty^\S$}\} \to
\{\text{irreducible components of $X_m^\S$}\}
\]
sending a non-degenerate irreducible component $C$ of $X_\infty^\S$
to the unique irreducible component $D$ of $X_m^\S$
containing the image of $C$ in $X_m$.
\end{theorem}

\begin{definition}
We say that an irreducible component of $X_m^\S$
is \emph{of arc type} if it is in the image of $\Psi_m^\S$.
\end{definition}

\begin{remark}
There are two special cases about \cref{t:welldefined-injective}. The first is when we take
$\S = \Sing X$. In this case every irreducible component of $X_\infty^{\Sing X}$
is non-degenerate and the domain of $\Psi_m^{\Sing X}$
is the set of Nash families of arcs.
The second special case is when $\S = X$. 
In this case, the domain of $\Psi_m^X$ is a singleton
and the image of $\Psi_m^X$ is the irreducible component of $X_m$
dominating $X$, namely, the closure of $(X_\reg)_m$. 
\end{remark}

We will break the proof of \cref{t:welldefined-injective} into two steps:
proving that $\Psi_m^\S$ is well-defined, and showing that it is injective.
We may assume that $\S$ is nonempty, the statement being trivial otherwise.

We start with the basic observation that 
\[
\ff_m(X_\infty^\S) \subset X_m^\S.
\] 
This implies that for every non-degenerate irreducible component $C$ of $X_\infty^\S$
there exists an irreducible component $D$ of $X_m^\S$ such that
$\ff_m(C) \subset D$. 
Our goal is to prove that if $m \ge \m + \n$ then such component $D$ is unique
(proving well-definedness), and 
that a different component $D$ of $X_m^\S$ occurs for each 
non-degenerate component $C$ of $X_\infty^\S$
(proving injectivity). 

These properties follow by standard facts about the structure of 
the truncation maps, specifically from
Greenberg's theorem on liftable jets \cite{Gre66}
and from a result of Looijenga on the fibers of the 
truncation maps between jet schemes \cite{Loo02}. 
For convenience, we will cite these results from \cite{EM09}. 

We start with the first assertion.

\begin{lemma}
\label{p:welldefined}
If $m \ge \m + \n$, then for every non-degenerate irreducible component $C$ of $X_\infty^\S$
there exists a unique irreducible component $D$ of $X_m^\S$ 
such that $\ff_m(C) \subset D$. 
\end{lemma}

\begin{proof}
We proceed by contradiction and assume that there exists an 
integer $m \ge \m + \n$
and a non-degenerate irreducible component $C$ of $X_\infty^\S$ 
such that $\ff_m(C)$ is contained in the intersection of two distinct 
irreducible components $D$ and $D'$ of $X_m^\S$. 
Whatever the value of $m$, 
we can find another integer $n$ such that
\begin{enumerate}
\item
$n \ge \n$ and
\item
$2n \ge m \ge \m + n$. 
\end{enumerate}
A choice of $n$ can be made by setting $n = \n + k$ where $k$ is defined by $m = \m + \n + k$.

Let $\a \in C$, $\b \in D$ and $\b' \in D'$ denote the respective generic points, 
and let $\a_n = \ff_n(\a)$, $\b_n = \p_{m,n}(\b)$, and $\b_n' = \p_{m,n}(\b')$
be their images in $X_n$. 
Note that both $\b_n$ and $\b_n'$ specialize to $\a_n$. 
Since $\ord_\a(\Jac_X) \le \m \le n$, we have
\[
\ord_{\b_n}(\Jac_X) \le \ord_{\a_n}(\Jac_X) = \ord_\a(\Jac_X) \le \m \le n,
\]
hence \cite[Proposition~4.1(i)]{EM09} implies that $\b_n = \ff_n(\g)$ 
for some arc $\g \in X_\infty$. Similarly, we have $\b_n' = \ff_n(\g')$ 
for some $\g' \in X_\infty$.

Note that $\g, \g' \in X_\infty^\S$. In fact, as $n \ge \n$, 
we see that $\g, \g' \in C$ since, by the definition of $\n$, no other irreducible component
of $X_\infty^\S$ contains a point whose image in $X_m$ specializes to $\a_m$. 
In particular, $\g$ and $\g'$ are specializations of $\a$, 
hence $\b_n$ and $\b_n'$ are both generalizations and specializations
of $\a_n$, meaning that 
\[
\b_n = \a_n = \b_n',
\]
This means that $\b$ and $\b'$ belong to the same fiber of $X_m \to X_n$, 
namely, $\p_{m,n}^{-1}(\a_n)$. 

As $\a_n \in X_n^\S$, the fiber $\p_{m,n}^{-1}(\a_n)$ is contained in $X_m^\S$, 
and since it contains the generic points $\b$ and $\b'$ of the irreducible components
$D$ and $D'$ of $X_m^\S$, it follows that $D$ and $D'$ are irreducible components
of $\p_{m,n}^{-1}(\a_n)$. This contradicts the fact that, 
by \cite[Proposition~4.4(ii)]{EM09}, this fiber is irreducible. 
\end{proof}

We now turn to the second assertion.

\begin{lemma}
\label{p:injective}
If $m \ge \m + \n$, then for every irreducible component $D$ of $X_m^\S$ 
there exists at most one non-degenerate irreducible component $C$ of $X_\infty^\S$
such that $\ff_m(C) \subset D$. 
\end{lemma}

\begin{proof} 
We need to prove that if $m \ge \m + \n$ and $\a,\a' \in \G$ are such that
their images $\a_m$ and $\a_m'$ in $X_m$ belongs to the same 
irreducible component $D$ of $X_m^\S$, then $\a = \a'$. 

To prove this, let $\b \in D$ be 
the generic point. Then $\b$ specializes to both $\a_m$ and $\a_m'$, 
hence its image $\b_{m-\m} := \p_{m,m-\m}(\b) \in X_{m-\m}$ 
specializes to both images $\a_{m-\m}$ and $\a_{m-\m}'$ of $\a$ and $\a'$
in $X_{m-\m}$.
Note that $m-\m \ge \n \ge \m$. By semicontinuity, 
\[
\ord_{\b_{m-\m}}(\Jac_X) \le \m
\]
Then, by \cite[Proposition~4.1(i)]{EM09}, we see that $\b_{m-\m}$
lifts to an arc; that is, there exists $\g \in X_\infty$ such that 
$\ff_{m-\m}(\g) = \b_{m-\m}$. 
By construction, $\g \in \p^{-1}(\S)$, 
hence there exists $\a'' \in \G$ specializing to $\g$.
It follows that the image of $\a''$ in $X_{m-\m}$ 
specializes to both $\a_{m-\m}$ and $\a_{m-\m}'$.
As $m - \m \ge \n$, we conclude that $\a = \a'' = \a'$. 
\end{proof}

\begin{proof}[Proof of \cref{t:welldefined-injective}]
\cref{p:welldefined} implies that $\Psi_m^\S$ is well-defined for $m \ge \m + \n$, 
and \cref{p:injective} that this map is injective. 
\end{proof}

\begin{remark}
The definition of the function $\Psi_m^\S$ 
constructed in \cref{t:welldefined-injective}
can be extended to all $m \ge 0$ as long as one is willing to regard them
as multivalued function, sending each 
$C$ to all components $D$ containing the image of $C$. 
\end{remark}

\section{The question of surjectivity}

Given \cref{t:welldefined-injective},
it is natural to ask under which conditions on 
singularities one can guarantee that the maps $\Psi_m^\S$ are surjective.
These functions are well-defined for $m \gg 1$, but 
if we are willing to regard them as a multivalued functions, then we can remove
the constrain on $m$. The question of surjectivity still
makes sense for multivalued functions. 

Before we move to discuss the case we will be focusing on, 
it may be instructive to point out that there is already
an interesting answer to the problem (a sufficient condition 
for surjectivity) in the special case where $\S = X$.
This comes from Musta\c t\u a's theorem on 
locally complete intersection canonical singularities.

\begin{theorem}[\cite{Mus01}]
Let $X$ be a locally complete intersection variety with canonical singularities.
Then $\Psi_m^X$ is well defined and surjective for every $m$. 
\end{theorem}

\begin{proof}
As $X_\infty$ has only one non-degenerate irreducible component
(and in fact only one irreducible component
since it is irreducible by Kolchin's theorem \cite{Kol73}), 
this is just a restatement of Musta\c t\u a's theorem on the irreducibility of 
the jet schemes, since any additional irreducible component
of $X_m$ would lie over the singular locus of $X$ and therefore would not contain
the image of $X_\infty$. 
\end{proof}

Like in Musta\c t\u a's theorem, we will be focusing on 
locally complete intersection canonical singularities.
Our goal is to find a class of singularities
for which $\Psi_m^{\Sing X}$ is surjective. 
%This will apply in particular to prove the surjectivity of $\Psi_m^{\Sing X}$
%when $X$ has isolated singularities. 

To get a sense of what one can expect,  
we start by reviewing some cases that are already understood. 

\begin{example}[Nodal curve]
The case where $X$ is a nodal curve already shows
that one cannot expect $\Psi_m^{\Sing X}$ to be always surjective.
Indeed, if $x \in X$ is a node, then for $m \ge 3$ the set $X_m^x$ has $m-1$ irreducible
components, and only two of them are in the image of $\Psi_m^x$.
\end{example}

\begin{example}[Affine cones]
Let $V \subset \P^{N-1}$ be a smooth complete intersection
variety defined by equations of degree $r$, let 
$X \subset \A^N$ be the affine cone over $V$, and let $x \in X$ be the vertex.
As the blow-up of $x$ gives a resolution of $X$ with a single 
exceptional divisor, one easily see that $X_\infty^x$ is irreducible.
On the other hand, for every $m \ge r$ we have 
\[
\p_m^{-1}(x) \cong X_{m-r} \times \A^{N(r-1)},
\]
see, e.g., the proof of \cite[Theorem~3.5]{dFEM03}.
By \cite[Theorem~0.1]{Mus01}, we know that
if $X$ is canonical then $X_m$ is irreducible for all $m$, 
and conversely, using also \cite[Proposition~1.6]{Mus01}, 
if $X$ is not canonical at $x$ then $X_m$ is reducible 
for all $m \gg 1$. It follows that $X_m^x$ is irreducible
(hence $\Psi_m^x$ is surjective) for all $m \ge r$ if $X$ is canonical, and is reducible
(hence $\Psi_m^x$ fails to be surjective) for all $m \gg 1$ if $X$ is not canonical.
\end{example}

Mourtada, in part in collaboration with Pl\'enat and Cobo, 
has studied the irreducible decomposition of $X_m^{\Sing X}$ in many 
explicit situations where $X$ is a surface
%, and looked at the graph generated by the 
%irreducible components of $X_m^{\Sing X}$ for all $m \ge 0$, 
%where edges are drawn 
%whenever an irreducible component of $\p_{m+1}^{-1}(\Sing X)$ 
%maps into an irreducible component of $\p_m^{-1}(\Sing X)$ under the truncation map $X_{m+1}\to X_m$
\cite{Mou14,Mou17,MC18,CM21};
see also \cite{Kor22} for related work. 
While in some cases these results indicate that
the number of components continues to grow with $m$, there are also
cases where
the number of components stabilizes and matches the number of Nash families. 

%The next example shows that surjectivity can fail for normal singularities.

\begin{example}[Toric surface singularities]
The irreducible decomposition of $X_m^{\Sing X}$ was computed 
for toric surfaces by Mourtada \cite{Mou17}, and the only case where we have 
the same number of components as Nash families is when $X$ has $A_n$-singularities.
\end{example}

\begin{example}[Du Val singularities]
It is proved in \cite{Mou14} that, for $m \gg 1$, 
the number of families of $m$-jets through a Du Val singularity coincides with the number of exceptional
divisors on the minimal resolution, hence with the number of Nash families of arcs.
It follows in particular that in this case $\Psi_m^{\Sing X}$ is a bijection. 
\end{example}

\begin{example}[cA-type singularities]
Another case where we can check directly that $\Psi_m^{\Sing X}$ is a bijection
is that of $cA$-type singularities. Nash families of arcs on these singularities
were described in \cite{JK13}, and the deformation argument used in their proof
can be adapted to show that, for $m \gg 1$, there is the same number of families of $m$-jets
through the singularity, proving that $\Psi_m^{\Sing X}$ is a bijection in this case as well.
More specifically, suppose $X$ is defined by an equation
\[
xy = f(z_1,\dots,z_{d-1})
\]
in $A = \A^{d+1}$, where $\m := \mult_0(f) \ge 2$. 
The proof in \cite{JK13} begins by identifying 
$\m - 1$ irreducible open sets $U_i \subset X_\infty^0$, 
for $1 \le i \le \m - 1$, given by
\[
U_i = \{ \a \in X_\infty^0 \mid 
\ord_\a(x) = i,\, \ord_\a(y) = \m - i,\, \ord_\a(f)= \m \}.
\]

The proof then goes by showing that every arc $\a \in X_\infty^0$
can be deformed (in $X_\infty^0$) to an arc $\a^*$ with $\ord_{\a^*}(f)= \m$. 
Clearly such arc must belong to one of the $U_i$, hence proving that
the closures of these sets give all irreducible components of $X_\infty^0$. 
The deformation is done in several steps: first, one
deforms $\a$ to an arc $\a'$ with $\ord_{\a'}(f) < \infty$, 
and if $\ord_{\a'} > \m$, then one deforms $\a'$ to
an arc $\a''$ with $\ord_{\a''}(f) < \ord_{\a'}(f)$. 
After a finite number of steps, this process produces the desired arc $\a^*$. 

This argument can be adapted to characterize the irreducible components of 
$X_m^0$, for any given $m \ge \m$, as follows. 
For $1 \le i \le \m - 1$, we consider the irreducible open sets
\[
V_i = \{ \b \in X_m^0 \mid 
\ord_\b(x) = i,\, \ord_\b(y) = \m - i,\, \ord_\b(f)= \m\}.
\]
Given any $\b \in X_m^0$, we take any lift $\a \in A_\infty^0$
(i.e., any arc $\a$ on $A$ such that $\psi_m^A(\a) = \b$)
and apply the same deformation argument as in \cite{JK13}
to produce a new arc $\a^* \in A_\infty^0$ such that
$\ord_{\a^*}(f) = \m$. In fact, without loss of generality
we can pick $\a$ so that $\ord_\a(f) < \infty$, 
hence skip the first deformation
and just deform to reduce $\ord_\a(f)$ if the order of contact 
is larger than $\m$. 
The key observation here is that, just like in \cite{JK13}
the deformation keeps the arc on $X$,
in this setting the deformation maintains the order of contact of the arc with $X$, hence the corresponding deformation at level $m$ stays on $X_m$. 
\end{example}

The above examples are mainly understood through their equations. 
Our goal is to identify a new class of examples of arbitrary dimensions
where $\Psi_m^{\Sing X}$ is surjective, without having to rely on
explicit equations. This will be done in the next two sections.

\section{Singularities of maximal embedding codimension}

For a local ring $(R,\fm)$ we denote by $\dim(R)$ the {Krull dimension},
by $\embdim(R)$ the {embedding dimension}
(the dimension of the Zariski tangent space)
and by $\embcodim(R)$ the {embedding codimension}
(the codimension of the tangent cone in the Zaristi tangent space).
When $R$ is Noetherian, the latter is also known as the 
{regularity defect} \cite{Lec64} and is equal to $\embdim(R) - \dim(R)$.

We start by establishing the following bound on embedding codimension
for normal locally complete intersection singularities.
The bound is likely known to experts. 

\begin{proposition}
\label{p:mld-bound}
Let $X$ be a normal locally complete intersection variety. Then
\[
\embcodim(\O_{X,x}) \le \dim(\O_{X,x}) - \mld_x(X)
\] 
for every $x \in X$.
\end{proposition}

\begin{proof}
The assertion being trivial if $\mld_x(X) = -\infty$, we assume that
$\mld_x(X) \ge 0$.  
Working locally in $X$, we may assume that $X$ is embedded in an affine space $A := \A^N$.
Let $d = \dim(X)$, $r = \dim(\O_{X,x})$, $e = \embcodim(\O_{X,x})$ 
and $c = \codim(X,A)$. 
By inversion of adjunction (see \cref{t:inv-adj}), 
\[
\mld_x(X) = \mld_x(A,cX).
\]

Let $\fm_x \subset \O_{X,x}$ be the maximal ideal.
By applying \cite[Theorem~25.2]{Mat89} to the sequence
$k \to \O_{X,x} \to k_x$, we get the exact sequence
\[
0 \to \fm_x/\fm_x^2 \to \Om_{X/k} \otimes k_x \to \Om_{k_x/k} \to 0.
\]
This gives
\[
\dim_{k_x}(\Om_{X/k}\otimes k_x) = \embdim(\O_{X,x}) + d-r = d + e.
\]
By the isomorphism $X_1 \cong \Spec(\Sym(\Om_{X/k}))$
(see \cite[Example~2.5]{EM09} or \cite[(1.4)]{Voj07}), 
%$\Om_{X/k}\otimes k_x$ and the set $X_1^x$ of 1-jets stemming from $x$, 
we have $X_1^x \cong \Spec(\Sym(\Om_{X/k} \otimes k_x))$,
%\[
%\dim_{k_x}(X_1^x) = \dim_{k_x}(\Om_{X/k}\otimes k_x) = d + \d,
%\]
hence
\[
\dim_k(\ov{X_1^x}) = \dim_{k_x}(X_1^x) + d - r  = 2d + e - r.
\]
%Furthermore, the closure $\ov{X_1^x} \subset X_1$ of $X_1^x$ is an irreducible closed set, so its 
The reduced inverse image $V \subset A_\infty$ 
of the closure $\ov{X_1^x} \subset A_1$ of $X_1^x$ 
%under the projection $A_\infty \to A_1$ 
is a closed irreducible cylinder.
% in $A_\infty$. 
Let $v$ be the valuation defined by
$V$ (namely, $v = \ord_\a$ where $\a \in V$ is the generic point).
By \cite[Theorem~C]{ELM04}, 
$v$ is a divisorial valuation, i.e., $v = p\ord_F$ where $F$
is a prime divisor over $A$ and $p$ is a positive integer. 
Note that, by construction, we have $v(\I_X) \ge 2$. 
If $C_A(v) \subset A_\infty$ is the maximal divisorial set associated
to the valuation, then we have $V \subset C(v)$, hence 
\[
\codim(V,A_\infty) \ge \codim(C_X(v),A_\infty) = p\,a_F(A)
\]
(the last formula is implicit in \cite{ELM04}; 
for a direct reference, see \cite[Theorem~3.8]{dFEI08}).
On the other hand, 
\begin{align*}
\codim(V,A_\infty) 
&= \codim(\ov{X_1^x},A_1) \\
&= \dim(A_1) - \dim(\ov{X_1^x})\\
&= 2(d+c) - (2d + e - r)\\
&= r - e - 2c.
\end{align*}
It follows that
\[
\mld_x(A,cX) \le a_F(A, cX) \le \frac 1p \big(\codim(V,A_\infty) - 2c\big) \le r-e,
\]
where we use in the last inequality that $\mld_x(X) \ge 0$ to ensure that
the inequality is preserved when we clear the denominator $p$. 
\end{proof}

\begin{definition}
\label{d:max-ecodim}
In accordance with \cref{p:mld-bound}, we say that a normal locally complete intersection variety
$X$ has \emph{maximal embedding codimension singularities} if
\[
\embcodim(\O_{X,x}) = \dim(\O_{X,x}) - \mld_x(X)
\] 
for every $x \in X$.
\end{definition}

\begin{remark}
Smooth varieties have maximal embedding codimension singularities.
\end{remark}

\begin{remark}
Every locally complete intersection variety with maximal embedding codimension singularities
has log canonical singularities, since the condition 
implies that $\mld_x(X) \ne - \infty$ hence $\mld_x(X) \ge 0$ for all $x \in X$.
Note that if $X$ is a curve then normality already implies that $X$ is smooth.
\end{remark}

\begin{example}[Hypersurface singularities]
A normal hypersurface singularity $x \in X$ has maximal embedding codimension 
if and only if $\mld_x(X) = \dim(\O_{X,x}) - 1$.
In particular Du Val singularities in dimension 2 and isolated 
cDV singularities of dimension 3 are all the examples in these dimensions
of isolated hypersurface singularities of maximal embedding codimension
(cf.\ \cite{Rei83}).
\end{example}

\section{Higher Du Val singularities}

We now identify a particular subclass of locally complete intersection 
varieties with maximal embedding codimension singularities
which can be thought as a higher dimensional version of Du Val singularities.

\begin{definition}
\label{d:crepant-sing}
Let $X$ be a normal locally complete intersection variety of dimension $d \ge 2$.
We say that a point $x \in X$ is a \emph{higher Du Val} (hDV) \emph{singularity}  if
\[
\mld_x(X) = \dim(\O_{X,x}) - \embcodim(\O_{X,x}) = 1.
\] 
\end{definition}

By definition, hDV singularities are canonical but not terminal.
They can be locally embedded as complete intersection
singularities of codimension $d-1$ in $\A^{2d-1}$ 
(cf.\ \cite[Theorem~3.15]{CdFD22})
but not in any smaller affine space.  
In dimension two, these are the same as the Du Val singularities.

\begin{remark}
It is useful to compare the above definition with another
classical way of generalizing Du Val singularities, namely,
compound Du Val singularities.
Compound Du Val singularities preserve two properties of Du Val singularities: 
being hypersurface singularities, 
and having minimal log discrepancy $\mld_x(X) = \dim(X) - 1$. 
By contrast, the definition of hDV singularities preserves the condition that $\mld_x(X) = 1$
and requires maximal embedding codimension.
The attribute ``higher'' in hDV singularity reflects at the same time that
these are higher dimensional and higher codimensional 
generalizations of Du Val singularities.
\end{remark}

\begin{remark}
If we extended \cref{d:crepant-sing} to the case $d=1$, then in dimension one the definition
would characterize smooth points on curves. This says something meaningful
about the behavior of this notion as a function of dimension. 
We prefer to assume $d \ge 2$ as we want to regard this as defining
a class of actual singular points.
\end{remark}

\begin{example}[Intersections of quadric cones]
In higher codimensions, the simplest example of a hDV singularity 
is the cone $X \subset \A^{2e+1}$ over the transversal intersection
of $e$ smooth quadrics in $\P^{2e}$. 
The blow-up of the vertex $x$ of the cone gives a log resolution
of $(\A^{2e+1},X)$, and
\[
\mld_x(X) = \mld_x(\A^{2e+1},eX) = 1
\]
where the minimal log discrepancy is computed by the exceptional divisor
of the blow-up.
\end{example}

More generally, we have the following set of examples, 
which shows the clear analogy with Du Val singularities. 

\begin{proposition}
\label{p:eg-hDV}
Let $e \ge 1$, let $(u_1,\dots,u_{2e-2},x,y,z)$ denote affine coordinates of $\A^{2e+1}$, 
and let $X \subset \A^{2e+1}$ be the subvariety defined by the vanishing of 
$e$ general linear combinations of any finite set of generators of
the ideal 
\[
\fa = (u_1,\dots,u_{2e-2})^2 + \fb
\]
of $k[u_1,\dots,u_{2e-2},x,y,z]$, where $\fb$ is one of the following:
\[
\fb = 
\begin{cases}
(x^2,y^2,z^{n+1}) \quad (n \ge 1) &\text{$A_n$-type} \\
(z^2,x^2y,y^{n-2}) \quad (n \ge 4) &\text{$D_n$-type} \\
(z^2,x^3,y^4) &\text{$E_6$-type} \\
(z^2,x^3,xy^3) &\text{$E_7$-type} \\
(z^2,x^3,y^5) &\text{$E_8$-type}
\end{cases}
\]
Then $X$ has a hDV singularity at the origin $0 \in \A^{2e+1}$. 
\end{proposition}

\begin{proof}
Clearly, $X$ is a complete intersection variety with an isolated singularity at the origin, 
and $\embcodim(\O_{X,0}) = e$. What is left to show is that $\mld_0(X) = 1$. 
Note that $\mld_0(X) = \mld_0(\A^{2e+1},eX)$. 
By looking at the exceptional divisor of
the blow-up of $\A^{2e+1}$ at the origin, we see that $\mld_0(\A^{2e+1},eX) \le 1$. 
On the other hand, a special case of the Thom--Sebastiani theorem (see \cite[Proposition~8.21]{Kol97})
gives us the following formula for the log canonical thresholds of $\fa$:
\[
\lct(\fa) = \lct((u_1,\dots,u_{2e-2})^2) + \lct(\fb) = e-1 + \lct(\fb).
\]
What we know about Du Val singularities already tells us that $\lct(\fb) > 1$; 
this can also be checked directly using Howald's formula for the log canonical 
threshold of monomial ideals \cite{How01}. 
Therefore $\lct(\fa) > e$, hence $\mld_0(\A^{2e+1},eX) > 0$. 
We conclude that $\mld_0(\A^{2e+1},eX) = 1$, as required. 
\end{proof}

\begin{remark}
Assuming $k = \C$, hDV singularities are closely related 
certain hypersurface singularities studied by Arnol'd \cite{Arn72}. 
These are isolated hypersurface singularities 
characterized by the property that their versal deformations 
only contain finitely many analytically inequivalent singularities, and are known 
as \emph{simple singularities}. 
They were classified in \cite{Arn72}; see also \cite[Example~(3.4)]{Bur74}. 
In the notation of \cref{p:eg-hDV}, for any $\fa$ 
(which, according to the proposition, corresponds to an example of a hDV singularity)
the vanishing of a general element $h \in \fa$ 
defines a simple singularity, and all simple singularities arise in this way. 
Conversely, the examples of hDV singularities provided by \cref{p:eg-hDV}
are complete intersections of simple singularities of the same type. 
\end{remark}

\begin{proposition}
\label{p:crepant-sing-isolated}
Let $X$ be a variety with hDV singularities.
Then $X$ has isolated singularities. 
\end{proposition}

\begin{proof}
Let $f \colon Y \to X$ be a log resolution
that is an isomorphism over $X_\reg$, 
and let $E$ be the reduced exceptional locus.
Note that $K_{Y/X} \ge 0$.

If $\dim(\Sing X) \ge 1$, then we can find
a closed point $x \in \Sing X$
such that $x$ is not the center of any component of $E$. 
On the other hand, $x \in f(E)$. 
Now, let $F$ be an arbitrary prime divisor over $X$ with center $c_X(F) = x$.
We may assume that $F$ lies on a nonsingular model $g \colon Z \to Y$. 
Since $f^{-1}(x)$ has codimension at least 2 in $Y$ 
and contains the center of $F$ in $Y$, we have $\ord_F(K_{Z/Y}) \ge 1$.
It follows that $\ord_F(K_{Z/X}) \ge 1$, hence $a_F(X) \ge 2$. 
This contradicts the fact that, by hypothesis, $\mld_x(X) = 1$.
\end{proof}

\begin{theorem}
\label{t:surjective}
Let $x \in X$ be a hDV singularity.
\begin{enumerate}
\item
\label{i2:surjective}
The multivalued function $\Psi_m^x$ is surjective for all $m$.
\item
\label{i3:surjective}
An irreducible set $C \subset X_\infty^x$ is a non-degenerate irreducible 
component if and only if $C = C_X(E)$ for some
prime divisor $E$ over $X$ with center $c_X(E) = x$
and log discrepancy $a_E(X) = 1$.
\end{enumerate}
\end{theorem}

\begin{proof}
By \cref{p:crepant-sing-isolated}, $x \in X$ is an isolated singularity. 

Let $d = \dim(X) = \dim(\O_{X,x})$ and $e = \embcodim(\O_{X,x})$. Note that, by our assumption, $e = d-1$. 
Though not strictly necessary, to simplify the notation we apply \cite[Theorem~3.15]{CdFD22} to
reduce to the case where $X$ is embedded in $A := \A^{d+e}$.

Let $f_1,\dots,f_e \in k[x_1,\dots,x_{d+e}]$ be local generators of the ideal of $X$ in $A$ at the point $x$.
For every $j \ge 1$, we denote by $f_i^{(j)}$ the $j$-th Hasse--Schmidt derivative of $f_i$. 
As $X_1^x = A_1^x$ (by our choice of embedding), the polynomials $f_i$ and $f_i'$ vanish 
identically on $A_1^x$, hence on $A_m^x$. 
Therefore, the ideal of $X_m^x$ in $A_m^x$ is generated by 
the elements $f_i^{(j)}$ for $1 \le i \le e$ and $2 \le j \le m$.
In particular, if $D$ is any irreducible component of $X_m^x$, then 
\[
\codim(D,A_m^x) \le e(m-1).
\]
Noticing that $\codim(A_m^x,A_m) = d+e = 2e+1$, it follows that
\[
\codim(D,A_m) \le e(m+1) + 1
\]

Let $V \subset A_\infty$ be the cylinder over $D \subset A_m$. This is a closed 
irreducible cylinder of codimension 
\[
\codim(V,A_\infty) = \codim(D, A_m) \le e(m+1) + 1.
\]
If $v = p \ord_F$ is the divisorial valuation defined by the generic point of $V$, then 
$V \subset C(v)$, hence
\[
\codim(V,A_\infty) \ge \codim(C_X(v),A_\infty) = p\,a_F(A). 
\]
Note that $v(\I_X) \ge m+1$. Then
\[
\mld_x(A,eX) \le \frac 1p \big(\codim(V,A_\infty) - e(m+1) \big) \le 1.
\]
Since by our assumption on the singularity we have $\mld_x(X) = 1$, 
and $\mld_x(X) = \mld_x(A,eX)$ by inversion of adjunction, 
it follows that all inequalities in the above formula are equalities, 
and in particular $V = C_A(v)$. 

We see from the proof of \cref{t:inv-adj} (see also \cref{r:inv-adj=}) that there is
a non-degenerate irreducible component $W$ of $V \cap X_\infty$.
Furthermore, any such component $W$ is equal to $C_X(E)$ for some prime divisor $E$ over $X$
with center $c_X(E) = x$ and log discrepancy $a_E(X) = 1$.
Note that $W \subset X_\infty^x$.

We may assume that $E$ is an exceptional divisor on a log resolution $f \colon X' \to X$ of $X$.
We apply \cite[Corollary 1.4.3]{BCHM10} to $X$ and $f$, with $\D = \D_0 = 0$ and $\mathfrak{E}$ equal to
the set of exceptional divisors with log discrepancy at most 1. 
The output of this operation is a terminal model $Y$ over $X$
where the center of $\val_E$ has codimension 1. 
This implies that $\val_E$ is a terminal valuation, hence, 
by \cite[Theorem~1.1]{dFD16}, a Nash valuation.

The fact that $W$ is the maximal divisorial set of a Nash valuation
implies that $W$ is an irreducible component of $X_\infty^x$. 
By construction, the image of $W$ in $X_m^x$ is contained in $D$, 
showing that $D$ is in the image of $\Psi_m^x$. 
This proves \eqref{i2:surjective}. 

To conclude, we use what we just proved and
the injectivity of $\Psi_m^x$ established in 
\cref{t:welldefined-injective} for $m \gg 1$ 
to infer that every non-degenerate irreducible component of $X_\infty^x$
is of the form $C_X(E)$ for some prime divisor $E$ over $X$
with center $c_X(E) = x$ and log discrepancy $a_E(X) = 1$.
Conversely, as explained above, \cite[Theorem~1.1]{dFD16}
implies that for every prime divisor $E$ over $X$ with center $c_X(E) = x$
and log discrepancy $a_E(X) = 1$, the
set $C_X(E)$ is an irreducible component of $X_\infty^x$.
This gives \eqref{i3:surjective}. 
\end{proof}

We apply this result to give a solution of the Nash problem for varieties
with hDV singularities. 

\begin{corollary}
\label{c:Nash-crepant-sing}
Let $X$ be a variety with hDV singularities.
For a divisorial valuation $\ord_E$ on $X$, the following are equivalent:
\begin{enumerate}
\item
$\ord_E$ is a Nash valuation.
\label{i1:Nash-crepant-sing}
\item
\label{i2:Nash-crepant-sing}
$\ord_E$ is a terminal valuation.
\item
\label{i3:Nash-crepant-sing}
$E$ is exceptional over $X$ and $a_E(X) = 1$.
\end{enumerate}
\end{corollary}

\begin{proof}
The implication \eqref{i3:Nash-crepant-sing} $\Rightarrow$ \eqref{i2:Nash-crepant-sing}
follows by \cite[Corollary 1.4.3]{BCHM10},
the implication \eqref{i2:Nash-crepant-sing} $\Rightarrow$ \eqref{i1:Nash-crepant-sing}
follows by \cite[Theorem~1.1]{dFD16},
and the implication \eqref{i1:Nash-crepant-sing} $\Rightarrow$ \eqref{i3:Nash-crepant-sing}
follows by \cref{t:surjective}.
\end{proof}

This result illustrates how this class of singularities
preserves some of the properties that characterize Du Val singularities. 
By \cite[Corollary 1.4.3]{BCHM10}, there is a terminal model $Y \to X$
whose exceptional locus consists exactly of the divisors
with log discrepancy 1 over $X$; from this perspective,
this model should be regarded as the analogue
of the minimal resolution of a Du Val singularity. 
Needless to say, it would be interesting to
further study the structure of these higher dimensional singularities.

\section{Higher compound Du Val singularities}

In this section, we look again at rational singularities of maximal embedding codimension.
We recall that these are normal, isolated, locally complete intersection singularities.
A particular example of such singularities is given by isolated compound Du Val singularities.
Compound Du Val singularities were originally introduced in dimension three in \cite{Rei83}. 
In general, they are defined as follows. 

\begin{definition}
\label{d:cDV}
We say that $x \in X$ is a \emph{compound Du Val} (cDV) \emph{singularity} if
the surface $S \subset X$ cut out by $\dim(X)-2$ general hyperplane sections through $x$
has a Du Val singularity at $x$. 
\end{definition}

The following property characterizes isolated cDV singularities 
(cf.\ \cite{Mar96} for an earlier result in this direction in dimension three).

\begin{proposition}
\label{p:cDV}
Let $x \in X$ be an isolated hypersurface singularity of dimension $d \ge 3$. Then the following are equivalent:
\begin{enumerate}
\item
\label{eq1:cDV}
$x \in X$ is a cDV singularity.
\item
\label{eq2:cDV}
$\mld_x(X) = d-1$, and for every divisor $E$ over $X$ computing $\mld_x(X)$ we have
$\ord_E(\fm_x) = 1$ and $E$ computes $\mld_x(X,(d-2)\{x\})$. 
\end{enumerate}
In particular, isolated cDV singularities
are normal locally complete intersection singularities of maximal embedding codimension,
according to \cref{d:max-ecodim}.
\end{proposition}

\begin{proof}
First note that if $x \in X$ is a normal locally complete intersection singularity, 
then, by \cref{p:mld-bound}, 
we have $\mld_x(X) \le d-1$ and $\ord_E(\fm_x) \ge 1$ for any divisor $E$ over $X$ with center $x$. 
On the other hand, if $S$ is cut out by $d-2$ general hyperplane sections through $x$, 
then $\mld_x(S) \le 1$, and $x \in S$ is a Du Val singularity if and only if $\mld_x(S) = 1$. 

Assume \eqref{eq1:cDV} holds. 
If $S$ is cut out by general hyperplane sections as in \cref{d:cDV}, then 
$\ord_E(\I_S) = \ord_E(\fm_x)$ for any $E$ computing $\mld_x(X)$ and
\[
1 = \mld_x(S) = \mld_x(X,(d-2)S) \le a_E(X,(d-2)S) = \mld_x(X) - (d-2) \ord_E(\fm_x)
\]
by inversion of adjunction (\cref{c:inv-adj}).
The properties listed in \eqref{eq2:cDV} follows easily from this inequality.

Conversely, if \eqref{eq2:cDV} holds and $E$ is any divisor computing $\mld_x(X)$, then we have 
\[
\mld_x(S) = \mld_x(X,(d-2)S) = a_E(X,(d-2)S) = a_E(X,(d-2)\{x\}) = 1,
\]
hence $S$ is a Du Val singularity. Here we used again that $S$ is cut out by general 
hyperplane sections through $x$, hence $\ord_E(\I_S) = \ord_E(\fm_x)$. 
\end{proof}

\cref{p:cDV} implies in particular that cDV singularities are examples of rational
singularities of maximal embedding codimension. However, they 
satisfy an additional property, namely, the condition that 
for every divisor $E$ over $X$ computing $\mld_x(X)$ we have
$\ord_E(\fm_x) = 1$ and $E$ computes $\mld_x(X,(d-2)\{x\})$. 
It is not clear to us whether this condition might follow from the definition
of singularity of maximal embedding codimension.

By regarding hDV singularities as a higher dimensional 
version of Du Val singularities, we extend the notion of cDV singularity in the following way.

\begin{definition}
We say that $x \in X$ is a \emph{higher compound Du Val} (hcDV) \emph{singularity} if, for some $r \ge 0$, 
the variety $Y \subset X$ cut out by $r$ general hyperplane sections through $x$
has a hDV singularity at $x$. 
(Alternatively, one could call these singularities \emph{compound higher Du Val singularities.})
\end{definition}

A straightforward adaptation of \cref{p:cDV} gives the following property. 

\begin{proposition}
\label{p:compound-max-ecodim}
Let $x \in X$ be an isolated locally complete intersection
singularity of dimension $d \ge 3$ and embedding codimension $e$. 
Then the following are equivalent:
\begin{enumerate}
\item
\label{eq1:hcDV}
$x \in X$ is a hcDV singularity.
\item
\label{eq2:hcDV}
$\mld_x(X) = d-e$, and for every divisor $E$ over $X$ computing $\mld_x(X)$ we have
$\ord_E(\fm_x) = 1$ and $E$ computes $\mld_x(X,(d-e-1)\{x\})$. 
\end{enumerate}
In particular, isolated hcDV singularities
are normal locally complete intersection singularities of maximal embedding codimension,
according to \cref{d:max-ecodim}.
\end{proposition}

\begin{theorem}
\label{t:cDV}
Let $x \in X$ be an isolated hcDV singularity. 
Then the function $\Psi_m^x$ is surjective, hence a bijection, for all $m \gg 1$.
\end{theorem}

\begin{proof}
With the case of hDV singularities already settled in \cref{t:surjective}, we may assume that $\mld_x(X) > 1$. 
Let $d = \dim(X)$ and $e = \embcodim(\O_{X,x})$. Note that $\mld_x(X) = d-e$.
As in the proof of \cref{t:surjective}, for simplicity we
reduce to the case where $X$ is embedded in $A := \A^{d+e}$.
Let $H := \A^{2e+1} \subset A$ a general linear subspace of codimension $d-e - 1$ through $x$, 
so that $Y := X \cap H$ is a variety with a hDV singularity at $x$. 

Let $m$ be any positive integer such that:
\begin{enumerate}
\item
\label{item-i}
\cref{t:welldefined-injective} holds for $Y$ (with $\S = \{x\}$), and
\item
\label{item-ii}
for every divisor $E$ over $X$ computing $\mld_x(X)$, we have
\[
d(m+1) - \dim(\ff_m^X(C_X(E))) = \jetcodim(C_X(E),X_\infty).
\]
\end{enumerate}
Note that these conditions hold for all $m \gg 1$. 
We can guarantee \eqref{item-i} because there are only finitely many divisorial valuations
computing $\mld_x(X)$ since the minimal log discrepancy is positive. 

Let $D$ be an irreducible component of $X_m^x$, and pick an irreducible component
$D'$ of $D \cap Y_m^x$. If $h_1,\dots,h_{d-e-1}$ are linear forms on $A$ cutting out $H$, 
then $D \cap Y_m^x$ is cut out off $D$ by the equations
$h_i^{(j)} = 0$ for $1 \le i \le d-e-1$ and $1 \le j \le m$, hence 
\[
\codim(D',D) \le (d-e-1)m.
\]
If $f_1=\dots=f_e = 0$ are local equations of $X$ at $x$ in $A$, then $X_m^x$ is cut out in $A_m^x$
by the equations $f_i^{(j)} = 0$ for $1 \le i \le e$ and $2 \le j \le m$. Here we are using that $X$
is singular at $x$ hence, for all $i$, both $f_i$ and $f'_i$ vanish identically on $A_m^x$. 
This implies that 
\[
\codim(D,A_m^x) \le e(m-1).
\]
Since $\codim(H_m^x,A_m^x) = (d-e-1)m$, we obtain 
\[
\codim(D',H_m^x) \le e(m-1),
\] 
hence 
\[
\codim(D',H_m) \le e(m+1) + 1.
\]

Let $V' \subset H_\infty$ the cylinder over $D'$. 
We have
\[
\codim(V',H_\infty) \le e(m+1) + 1.
\]
Write $\ord_{V'} = p'\ord_{F'}$ for some divisor $F'$ over $H$ and some positive integer $p'$.
The same argument as in the proof of \cref{t:surjective} implies
\[
1 = \mld_x(Y) = \mld_x(H,e Y) \le \frac 1p \big( \codim(V',H_\infty) - e(m+1) \big) \le 1.
\]
This implies that $p'=1$, $V' = C_H(F')$, and $F'$ computes $\mld_x(H,e Y)$.
If $W' \subset Y_\infty$ is any non-degenerate irreducible component of $V' \cap Y_\infty$,
then the argument also shows that $W'$ is an irreducible component of $Y_\infty^x$ and
it is equal to $C_Y(E')$ for some divisor $E'$ over $Y$ with $a_{E'}(Y) = 1$.
Furthermore, the argument implies that all inequalities above are equalities. 

In particular, if $V \subset A_\infty$ is the cylinder over $D$ then
\[
\codim(V,A_\infty) = e m+d.
\]
Writing $\ord_V = p\ord_F$ for some divisor $F$ over $A$ and
arguing again as in the proof of \cref{t:surjective} (using now that, by \cref{p:compound-max-ecodim}, 
$\mld_x(A,e X) = d-e$), we 
conclude that $V = C_X(F)$ where $F$ is a divisor over $A$ computing $\mld_x(A,X)$.
Moreover, there is an irreducible component $W$ of $V \cap X_\infty$
that is not contained in $(\Sing X)_\infty$, 
and this component is of the form $W = C_X(E)$ for a divisor $E$ over $X$ computing $\mld_x(X)$. 

By construction, 
\[
\ff_m^X(W) \subset D.
\]
We do not know, however, that $W$ is an irreducible component of $X_\infty^x$.
Note that we cannot apply \cite{dFD16} as we did in the proof of \cref{t:surjective}
(and, above, for $W'$) since now $E$ does not define a terminal valuation over $X$.
The claim is that $Z \subset X_\infty^x$ is any irreducible component containing $W$, then 
\[
\ff_m^X(Z) \subset D.
\]
This is all we need to conclude that $D$ is in the image of $\Psi_m^x$. 

To prove the claim, we proceed as follows. 
First, note that $W' \subset W \cap Y_\infty$.
As discussed above, we have $W = C_X(E)$ and $W' = C_S(E')$
where $E$ and $E'$ are divisors 
over $X$ and $S$, respectively, with center $x$ and log discrepancies
$a_E(X) = d - e$ and $a_{E'}(X) = 1$. In particular, 
\[
a_{E'}(X) = a_E(X) - (d - e - 1).
\]
Since $X$ and $S$ are locally complete intersections at $x$, we have
\begin{gather*}
a_E(X) = \^a_E(X) - \ord_E(\Jac_X), \\
a_{E'}(Y) = \^a_{E'}(Y) - \ord_{E'}(\Jac_Y)
\end{gather*}
by \cite[Corollary~3.5]{dFD14}.
By Teissier's Idealistic Bertini Theorem \cite[2.15~Corollary~3]{Tei77}, we have
$\ov{\Jac_Y} = \ov{\Jac_X|_Y}$
(the bar denoting integral closure), hence it follows by the inclusion $W' \subset W \cap Y_\infty$
that
\[
\ord_{E'}(\Jac_Y) \ge \ord_E(\Jac_X).
\]
Combining these formulas, we see that
\[
\^a_{E'}(Y) \ge \^a_E(X) - (d-e-1).
\]

By \cite{dFEI08} and the assumption \eqref{item-ii} on our choice of $m$, we have 
\begin{gather*}
\^a_E(X) = d(m+1) - \dim(\ff_m^X(W)),\\
\^a_{E'}(Y) \le (e+1)(m+1) - \dim(\ff_m^Y(W')).
\end{gather*}
Using the previous inequality, we get
\[
\dim(\ff_m^Y(W')) \le \dim(\ff_m^X(W)) - (d-e-1)n.
\]

Observe that $\ff_m^Y(W')$ is contained in $\ff_m^X(W) \cap Y_m^x$, which is cut out
from $\ff_m^X(W)$ by the equations 
$h_i^{(j)} = 0$ for $1 \le i \le d-e-1$ and $1 \le j \le m$. 
Here we are using that the polynomials $h_i$ already vanish on $X_m^x$, hence on $\ff_m^X(W)$.
It follows that
\[
\dim(\ff_m^Y(W')) = \dim(\ff_m^X(W)) - (d-e-1)m,
\]
and the $h_i^{(j)}$ form a regular sequence at the generic point of $\ff_m^Y(W')$. 

Now, let $Z$ be an irreducible component of $X_\infty^x$ containing $W$, and 
assume by contradiction that $\ff_m^X(Z) \not\subset D$. 
Then $\ff_m^X(Z)$ must be contained in another irreducible component of $X_m^x$. 
In particular, if $\~D$ denote the union of all irreducible components of $X_m^x$
containing $\ff_m^Y(W')$ and different from $D$, then 
%\[
%\ff_m^X(Z) \subset \~D.
%\]
%From the inclusions
\[
\ff_m^Y(W') \subset D \cap \~D.
\]
Note that $(D \cup \~D) \cap Y_m^x$ is the union of the irreducible components of $Y_m^x$
containing $\ff_m^Y(W')$. 
Since the elements $h_i^{(j)}$ form a regular sequence at each generic point of $D \cap \~D$
and cut out $Y_m^x$ on $X_m^x$, it follows that $(D \cup \~D) \cap Y_m^x$ must be reducible.
This means that $\ff_m^Y(W')$ is contained in more than one irreducible component of $Y_m^x$, 
contradicting \cref{t:welldefined-injective}, which is supposed to holds for $Y$
by our assumption \eqref{item-i} on $m$.

We conclude that $\ff_m^X(Z) \subset D$, as claimed. This finishes the proof of the theorem. 
\end{proof}

\section{The graph generated by families of jets}

Following \cite{Mou14,Mou17,CM21}, to any variety $X$ we 
associate a directed graph $\G_X$ as follows.

\begin{definition}
Given a variety $X$, let $\G_X$ be the directed graph whose vertices 
corresponds to the irreducible components of $X_m^{\Sing X}$
for $m \ge 0$; an edge is drawn from a vertex $v$ to a vertex $v'$
whenever $v$ and $v'$ correspond, respectively, to irreducible components
$D \subset X_m^{\Sing X}$ and $D' \subset X_{m+1}^{\Sing X}$
with $\p_{m+1,m}(D') \subset D$. 
We say that a vertex $v$ has \emph{order $m$}, and write $\ord(v) = m$, if
$v$ corresponds to an irreducible component of $X_m^{\Sing X}$.
The orientation is defined by the order of the vertices.
For every $m$, we denote by $\G_X^{\ge m}$ and $\G_X^{\le m}$
the subgraphs of $\G_X$ obtained by removing 
all vertices of order $< m$, respectively, $> m$. 
We call the \emph{root} of $\G_X$ the set of vertices of order zero.
For any vertex $v$ of $\G_X$, the \emph{branch} of $\G_X$ stemming from $v$ 
is the subgraph  $\G_X^{\ge v}$ obtained by removing all
vertices that are not reachable by $v$. 
\end{definition}

By construction $\G_X$ is a directed acyclic graph, that is, a directed graph with no directed cycles.
Due to the finiteness of the irreducible components
of $X_m^{\Sing X}$, this graph has finitely many vertices
of any given order. In particular, $\G_X^{\le m}$ is finite for every $m$. 

\begin{corollary}
\label{c:graph-crepant-sing}
Let $X$ be a variety with isolated hcDV singularities, and let
$\G_X$ be the associated graph.
\begin{enumerate}
\item
\label{i1:graph-crepant-sing}
\emph{(Root)}.
The root of $\G_X$ is in natural bijection with the singular points of $X$. 
Each root in contained in a distinct connected component of $\G_X$. 
\item
\label{i2:graph-crepant-sing}
\emph{(Finite branches)}.
There are no finite branches in $\G_X$ beyond a certain order.
That is, 
there is an integer $m_0$ such that for every vertex $v$ of $\G_X$ 
of order $\ord(v) \ge m_0$ and every $m \ge \ord(v)$, 
there exists a vertex $u$ of order $m$
%(which belongs to one of the chains identified in \eqref{i1:graph-crepant-sing})
that is reachable by $v$. 
\item
\label{i3:graph-crepant-sing}
\emph{(Infinite branches)}.
The infinite branches of $\G_X$ are in bijection with the Nash valuations on $X$.
More precisely, for $m \gg 1$, the subgraph $\G_X^{\ge m} \subset \G_X$ is a disjoint union of 
infinite chains whose vertices have increasing orders $m, m+1, m+2, \dots$.
The number of chains is the number of Nash valuations on $X$, and
each chain is in natural correspondence with a distinct Nash valuation.
\end{enumerate}
In particular, for $m \ge 1$
the number of irreducible components of $X_m^{\Sing X}$ is equal 
to the number of irreducible components of $X_\infty^{\Sing X}$, 
and the function $\Psi_m^{\Sing X}$ is a bijection.
\end{corollary}

\begin{proof}
Property \eqref{i1:graph-crepant-sing} is clear since
the vertices in the root of $\G_X$ corresponds to the
singular points of $X$, viewed as $0$-jets on $X$.
Properties \eqref{i2:graph-crepant-sing} and \eqref{i3:graph-crepant-sing} 
follow from \cref{t:welldefined-injective,t:cDV}, 
which establish that $\Psi_m^{\Sing X}$ is a bijection for $m \gg 1$. 
The correspondence is defined by associating to each chain of $\G_X^{\ge m}$ 
the unique irreducible component $C$
of $X_\infty^{\Sing X}$ such that for $n \ge m$
its image $\ff_n(C)$ is contained in the irreducible component 
of $X_n^{\Sing X}$ corresponding to the vertex of order $n$ in the given chain. 

Implicit in these arguments is the compatibility of the functions
$\Psi_m^{\Sing X}$ as $m$ varies. Specifically, in the range of application of
\cref{t:welldefined-injective}, if $D = \Phi_m^{\Sing X}(C)$ 
and $D' = \Phi_{m+1}^{\Sing X}(C)$, then it follows by the geometric definition
of these functions and their injectivity that $\p_{m+1,m}(D') \subset D$, 
hence the corresponding vertices $v$ and $v'$ are joined by an edge.
\end{proof}

\begin{remark}
Regarding part \eqref{i2:graph-crepant-sing} of \cref{c:graph-crepant-sing},
we should remark that bounded branching of arbitrary large
order does occur for other singularities (e.g., see \cite{Mou17,CM21}).
As for \eqref{i3:graph-crepant-sing}, 
one can visualize the correspondence as attaching 
one vertex at the end of each chain, with such vertex corresponding 
to the Nash component.
Thinking of the chain as consisting of the integers on $[m,\infty)$, 
with the intervals $[n,n+1]$ representing the edges, 
this is the same as adding $\infty$ to get $[m,\infty]$. 
Note that this extension of $\G_X$ is not a graph, since we want to
see its geometric realization as a connected set but there is no
edge ending at $\infty$.
\end{remark}

\end{document}